\documentclass{amsart}

\usepackage{amsmath}
\usepackage{amsfonts}
\usepackage{amsthm}
\usepackage{graphicx}
\usepackage{hyperref}
\usepackage{todonotes}

\newtheorem{theorem}{Theorem}[section]
\newtheorem{lemma}[theorem]{Lemma}
\newtheorem{proposition}[theorem]{Proposition}

\newtheorem{definition}[theorem]{Definition}
\newtheorem{remark}[theorem]{Remark}
\newtheorem{example}[theorem]{Example}

\newcommand{\R}{\mathbb{R}}

\newcommand{\F}{\mathcal{F}}
\newcommand{\D}{\mathcal{D}}
\newcommand{\Id}{\mathrm{id}}
\renewcommand{\d}{\mathrm{d}}

\renewcommand{\epsilon}{\varepsilon}
\newcommand*{\mailto}[1]{\href{mailto:#1}{\nolinkurl{#1}}}

\numberwithin{equation}{section}

\begin{document}
\title[Lipschitz metric for the Hunter--Saxton system]{A Lipschitz metric for conservative solutions of the two-component Hunter--Saxton system}

\allowdisplaybreaks

\author[A. Nordli]{Anders Nordli}
\address{Department of Mathematical Sciences\\ Norwegian University of Science and Technology\\ NO-7491 Trondheim\\ Norway}
\email{\mailto{andenord@math.ntnu.no}}

\date{\today}


\begin{abstract}
We establish the existence of conservative solutions of the initial value problem of the two-component Hunter--Saxton system on the line. Furthermore we investigate the stability of these solutions by constructing a Lipschitz metric.
\end{abstract}

\maketitle

\section{Introduction}\label{section intro}
The two-component Hunter--Saxton system, given by
\begin{subequations}
\label{hunter-saxton-system}
\begin{align}
u_t(x,t)+u(x,t)u_x(x,t) &= \frac 14\bigg(\int_{-\infty}^x(u_x(z,t)^2+\rho(z,t)^2)\:\d z\nonumber\\ &\quad -\int_x^{\infty}(u_x(z,t)^2+\rho(z,t)^2)\:\d z\bigg),\\
\rho_t(x,t) +(u(x,t)\rho(x,t))_x &= 0,
\end{align}
\end{subequations}
was derived by Pavlov as a model of non-dissipative dark matter \cite{P}. It can also be viewed as a high frequency limit of the two-component Camassa--Holm system describing water waves \cite{W10}. The system \eqref{hunter-saxton-system} is a generalization of the Hunter--Saxton equation
\begin{equation}
\label{hunter-saxton}
(u_t+uu_x)_x = \frac 12 u_x^2,
\end{equation}
introduced by Hunter and Saxton as a model of the director field of a nematic liquid crystal \cite{HS}. 

Here we prove global existence of conservative weak solutions of the initial value problem for \eqref{hunter-saxton-system} on the line, and construct a metric that renders the flow Lipschitz continuous. Previously Wunsch has proven existence of solutions of \eqref{hunter-saxton-system} in the periodic setting \cite{W,W10}, and global existence of dissipative solutions on the real line \cite{W10}. Dissipative and conservative solutions are two distinct ways to extend the solution past the time where classical solutions break down. Before we discuss the difference between conservative and dissipative solutions we will look at how solutions break down.

A common feature for weak solutions of \eqref{hunter-saxton-system} and the Hunter--Saxton equation \cite{HS} is that weak solutions may experience wave breaking, which means that $u_x$ tends pointwise to $-\infty$ in finite time while $u$ stays continuous. The phenomenon is illustrated in the following example.
\begin{example}
\label{example solution}
Let $t\in [0,2)$ and let the functions $u$ and $\rho$ be defined by
\begin{align*}
u(x,t) &=
	\begin{cases}
	-\frac 12t+1,& x\leq -\frac 14t^2+t-1,\\
	-\frac{1}{-\frac 12t+1}x,& -\frac 14t^2+t-1\leq x \leq 0,\\
	\frac{t}{\frac 12t^2+2}x, &	0\leq x \leq \frac 14 t^2+1,\\
	\frac 12 t, & \frac 14t^2+1 \leq x,
	\end{cases}\\
\rho(x,t) &=
	\begin{cases}
	0,& x\leq 0,\\
	\frac{1}{\frac 14t^2+1}, &0< x \leq \frac 14 t^2+1,\\
	0, & \frac 14 t^2+1<x.
	\end{cases}
\end{align*}
Then $(u,\rho)$ is a weak solution of \eqref{hunter-saxton-system} for $t\in[0,2)$. Note that $u_{x}(0,t)\rightarrow -\infty$ as $t\rightarrow 2^-$, which in particular means that wave breaking occurs. We can define the energy of the system at time $t$ to be given by
\begin{equation}
\int_{\R} \big(u_x^2(x,t)+\rho^2(x,t)\big)\:\d x = 2,
\end{equation}
which is constant in time, even up to the point $t=2$. The energy contained in the interval $-\frac 14 t^2 + t -1\leq x\leq 0$, given by $\int_{-\frac 14 t^2 + t -1}^0 (u_x^2+\rho^2)\:\d x = 1$, is also conserved. Thus a finite amount of energy is being concentrated in a single point as $t\rightarrow 2^-$.
\end{example}
\begin{figure}
\includegraphics[width=8cm]{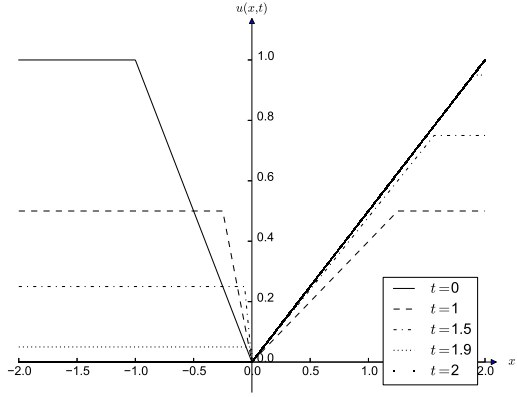}
\caption{A plot of $u$ in Example \ref{example solution} for $t=0,1,1.5,1.9,2$.}
\label{fig example intro}
\end{figure}
As seen in Example \ref{example solution} a part of the energy $\int_\R(u_x^2+\rho^2)\:\d x$ is focused at a single point at wave breaking. This illustrates that the energy density is not absolutely continuous, but a finite Radon measure in general. Nevertheless, the total energy remains constant in time as $t\rightarrow 2^-$. Hence $u_x,\rho$ stay in $L^2(\R)$ even if $u_x$ tends to minus infinity. This means that the energy can be described by the cumulative distribution function of a finite Radon measure. One can extend local solutions to global solutions by manipulating the concentrated energy at wave breaking. There are at least two ways to extend the solution to a global one past wave breaking. On the one hand one could ignore the part of the energy that concentrates on a set of measure zero in the continuation, which yields dissipative solutions.  On the other hand one could continue by letting the concentrated energy back into the system, which would give conservative solutions. In practice that would amount to defining $u$ and $\rho$ by the formulas in Example \ref{example solution} even for $t>2$. Thus it is essential to include the energy variable in our sets of variables, when constructing global conservative solutions.

We are going to solve the system \eqref{hunter-saxton-system} by the generalized method of characteristics. The approach is similar to the one by Dafermos in \cite{D}, where uniqueness of dissipative weak solutions of the Hunter--Saxton equation has been established. To that end we map our Eulerian coordinates $(u,\rho,(u_x^2+\rho^2)\:\d x)$ to the Lagrangian variables $(y,U,H,r)$ defined as follows. Let $y(\xi,t)$ be defined by $y_t(\xi,t) = u\big(y(\xi,t\big),t)$, $U(\xi,t) = u\big(y(\xi,t),t\big)$, and define $H(\xi,t) = \int_{-\infty}^{y(\xi,t)}\left(u_x(x,t)^2+\rho(x,t)^2\right)\d x$ as the energy to the left of $y(\xi,t)$. Note that $(u_x^2+\rho^2)\d x$ could be a singular measure. We introduce $r(\xi,t) = \rho\big(y(\xi,t),t\big)y_{\xi}(\xi,t)$. In Section \ref{section mappings} we will rigorously define the proper space for the variables $(y,U,H,r)$, and establish mappings between that space and the space for conservative solutions of \eqref{hunter-saxton-system}. In the above variables the Hunter--Saxton system reduces to
\begin{subequations}
\label{lagrangian system}
\begin{align}
y_t &= U,\\
U_t &= \frac 12 H-\frac 14H_{\infty},\\
H_t &= 0,\\
r_t &= 0,
\end{align}
\end{subequations}
where $H_{\infty} = \lim_{\xi\rightarrow\infty}H(\xi,0)$. The time evolution of $H$ follows from the conservation law
\begin{equation}
\label{conservation law}
(u_x^2+\rho^2)_t +(u(u_x^2+\rho^2))_x = 0,
\end{equation}
see for instance \cite{BHR,HZ3} for the similar conservation law for the Hunter--Saxton equation. In Section \ref{section existence} we solve \eqref{lagrangian system}, and together with the mappings from Section \ref{section mappings} prove that the we can construct global conservative solutions of \eqref{hunter-saxton-system}.

In Section \ref{section metric} we construct a Lipschitz metric. The idea is to construct the metric in the transformed variables $(y,U,H,r)$ to avoid having to deal with the measure. The Eulerian variables are one-to-one to equivalence classes of Lagrangian variables. We construct a functional $J$ that respects the equivalence structure such that it can be used as a building block of the metric. The approach here is thus more similar to the one employed for the Camassa--Holm equation \cite{LipCHline,LipCH} than the methods previously used to construct Lipschitz metrics for the scalar Hunter--Saxton equation \cite{BC,BHR}. 

\section{Mappings between Eulerian and Lagrangian coordinates}\label{section mappings}
In this section we define the sets of Lagrangian and Eulerian coordinates, and investigate the mappings between them. We introduce first an important ambient vector space $B$.
\begin{definition}
Let $E_1$ be the vector space defined by
\begin{equation}
E_1 = \{ f\in L^{\infty}(\R) \mid f'\in L^2(\R) \text{ and }\lim_{x\rightarrow -\infty}f(x) = 0\},
\end{equation}
equipped with the norm $\|f\|_{E_1} = \|f\|_{\infty}+\|f'\|_{2}$, and $E_2$ be defined by
\begin{equation}
E_2 = \{ f\in L^{\infty}(\R) \mid f'\in L^2(\R)\},
\end{equation}
equipped with the norm $\|f\|_{E_2} = \|f\|_{\infty}+\|f'\|_{2}$. Then define the normed space $B$ by $B = E_2\times E_2\times E_1\times L^2(\R)$, with the norm
\begin{equation}
\|(f_1,f_2,f_3,f_4)\|_B = \|f_1\|_{E_2} + \|f_2\|_{E_2} + \|f_3\|_{E_1} + \|f_4\|_2.
\end{equation}
\end{definition}
The natural space to look for solutions in Eulerian variables is the following.
\begin{definition}
\label{def D}
The space $\D$ consists of all triples $(u,\rho,\mu)$ such that
\begin{align*}
(i)&\;\;\; u\in E_2,\\
(ii)&\;\;\; \rho\in L^2(\R),\\
(iii)&\;\;\; \mu \in \mathcal{M}^+(\R),\\
(iv)&\;\;\; \mu_{ac} = (u_x^2+\rho^2)\:\d x,
\end{align*}
where $\mathcal{M}^+(\R)$ denotes the set of positive, finite Radon measures on $\R$.
\end{definition}
We are now ready to define the Lagrangian coordinates as a subset of $B$. The definition is similar to  \cite[Definition 2.2] {BHR}.
\begin{definition}
\label{def_F}
The set $\mathcal{F}$ consists of all quadruples $X=(y,U,H,r)$ such that $(y-\Id,U,H,r) \in B$, and there exists a number $c>0$ such that
\begin{align*}
(i)&\;\;\; y-\Id,U,H \in W^{1,\infty}(\R), r \in L^{\infty}(\R),\\
(ii)&\;\;\; y_{\xi}\geq 0, H_{\xi}\geq 0, H_{\xi}+y_{\xi} \geq c > 0 \text{ a.e.},\\
(iii)&\;\;\; y_{\xi}H_{\xi} = U_{\xi}^2 + r^2\;\;\;a.e.
\end{align*}
We define the subset $\F_0$ of $\F$ by
\begin{equation}
\F_0 = \{X\in\F\mid y+H = \Id\}.
\end{equation}
\end{definition}
We will use the notation $H_{\infty}=\lim_{\xi\rightarrow\infty}H(\xi) = \|H\|_{\infty}$. To be able to work with the space $\F$ we need a mapping from Eulerian to Lagrangian variables, and vice versa. Before the mappings are introduced, we will state a useful lemma on monotone Lipschitz continuous functions.
\begin{lemma}[{\cite[Lemma 3.9]{HR07}}]
\label{lipschitz lemma}
Let $f:\R\rightarrow\R$ be an increasing Lipschitz continuous function. Then for any set $B$ with $\mathrm{m}(B) = 0$, we have that $f_{\xi} = 0$ almost everywhere in $f^{-1}(B)$.
\end{lemma}
The map $L$ in the following definition maps $\D$ into $\F$, and thus represents a way to pass from Eulerian to Lagrangian coordinates in a rigorous manner. The definition is similar to \cite[Theorem 4.9]{GHR}.
\begin{definition}
\label{definition L}
Let the mapping $L:\D\rightarrow\F_0$ be defined by $L(u,\rho,\mu) = (y,U,H,r)$ where
\begin{subequations}
\begin{align}
\label{definition of y from L} 
y(\xi) &= \sup\{x|\mu((-\infty,x))+x < \xi\},\\
H(\xi) &= \xi - y(\xi),\\
U(\xi) &= u \circ y(\xi),\\
r(\xi) &= (\rho\circ y(\xi))y_{\xi}(\xi).
\end{align}
\end{subequations}
\end{definition}
\begin{proposition}
The mapping $L$ from Definition \ref{definition L} is well defined.
\end{proposition}
\begin{proof}
The proof follows closely those of \cite[Theorem 3.8]{HR07} and \cite[Theorem 4.9]{GHR}, but here the spaces $\D$ and $\F$ are different. The difference is that in our case we have $u,U\in L^{\infty}(\R)$, while \cite[Theorem 4.9]{GHR} uses $u,U\in L^2(\R)$. To prove that $L$ is well defined, let $(u,\rho,\mu)\in\D $ and define $X=(y,U,H,r) = L(u,\rho,\mu)$. We only prove that $U$ is in $L^{\infty}(\R)$. Since $u\in L^{\infty}(\R)$, and $y$ and $U$ are well defined functions, it holds that $U=u\circ y$ is bounded by $\|u\|_{\infty}$. 
\end{proof}
We compute an example to illustrate how the mapping $L$ works.
\begin{example}
\label{basic example}
Let $(u,\rho,\mu)\in\D$ be defined by
\begin{align}
u(x) &= \begin{cases} 1,& x<-1,\\ -x,&-1\leq x < 0,\\ 0, & 0\leq x,\end{cases}\\
\rho(x) &= \mathbf{1}_{[0,1]}(x),\\
\mu &= \mu_{ac}+\frac 12\delta_0.
\end{align}
The distribution function $F$ of the measure $\mu$ is given by
\begin{equation}
\label{distribution function}
F(x) = \mu\left((-\infty,x]\right) = 
\begin{cases}
0, & x\leq -1,\\
x+1, & -1\leq x < 0,\\
x+ \frac 32, & 0\leq x\leq 1,\\
\frac 52, & 1\leq x.
\end{cases}
\end{equation}
Then $(y,U,H,r)=L(u,\rho,\mu)$ is given by
\begin{subequations}
\label{basic example equations}
\begin{align}
y(\xi) &=
	\begin{cases}
	\xi, & \xi\leq -1,\\
	\frac 12(\xi-1),& -1\leq \xi\leq 1,\\
	0,& 1\leq\xi\leq \frac 32,\\
	\frac 12(\xi-\frac 32), & \frac 32\leq \xi \leq \frac 72,\\
	\xi-\frac 52,& \frac 72 \leq \xi,
	\end{cases}\\
U(\xi) &=
	\begin{cases}
	1, & \xi\leq -1\\
	-\frac 12(\xi-1), &-1\leq \xi\leq 1,\\
	0, & 1\leq \xi,
	\end{cases}\\
H(\xi) &= 
	\begin{cases}
	0, & \xi\leq -1,\\
	\frac 12(\xi+1),& -1\leq \xi\leq 1,\\
	\xi,& 1\leq\xi\leq \frac 32,\\
	\frac 12(\xi+\frac 32), & \frac 32\leq \xi \leq \frac 72,\\
	\frac 52,& \frac 72 \leq \xi,
	\end{cases}\\
r(\xi) &=\frac 12\mathbf{1}_{[\frac 32,\frac 72]}.
\end{align}
\end{subequations}
See Figure \ref{figure example} for a plot of the functions $F$, $y$, and $H$. Note that the Dirac delta corresponds to a flat interval in $y$, with the length of the interval equal to the strength of the delta.
\end{example}
\begin{figure}
\centering
\label{figure example}
\includegraphics[width=8cm]{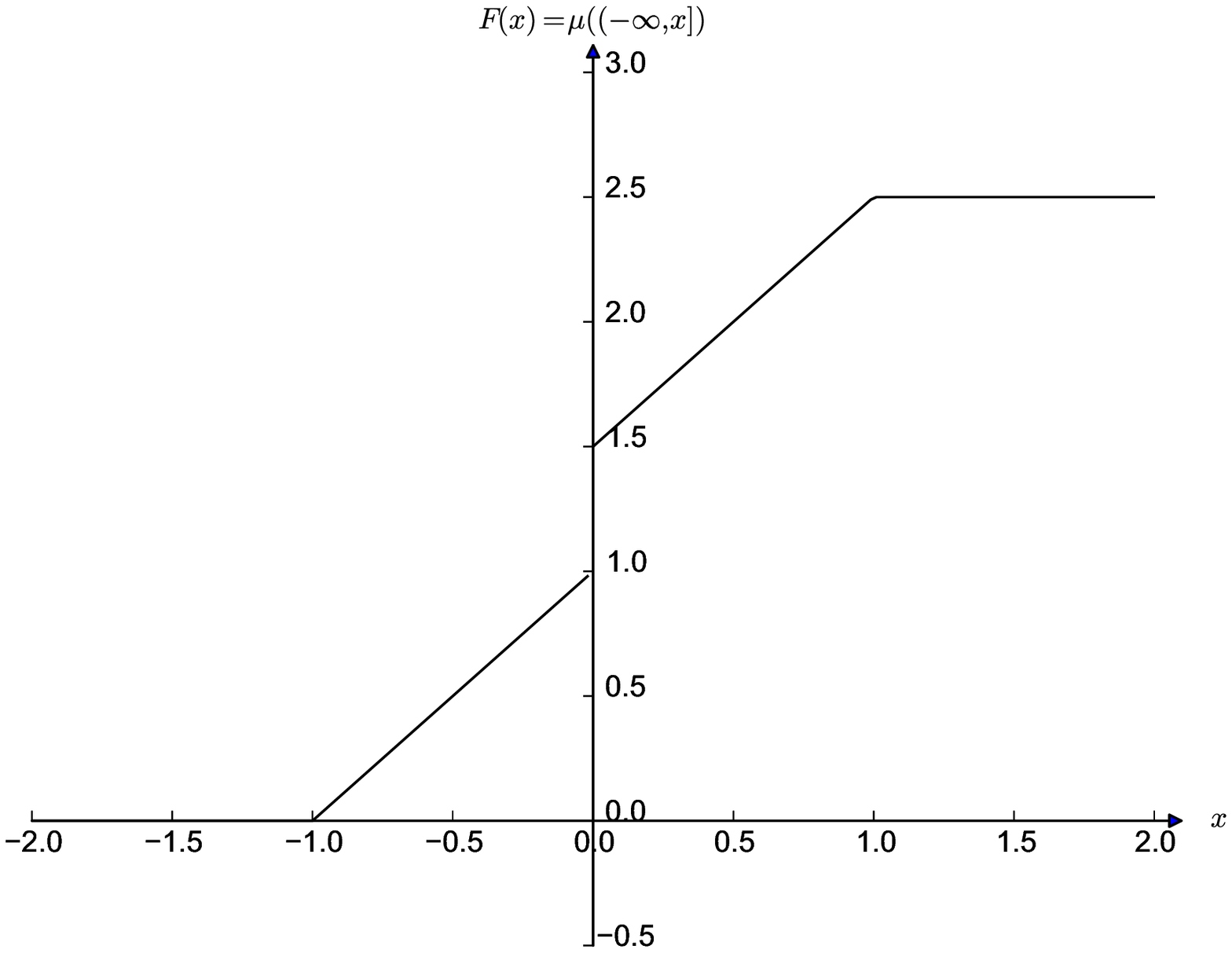}\\
\includegraphics[width=6cm]{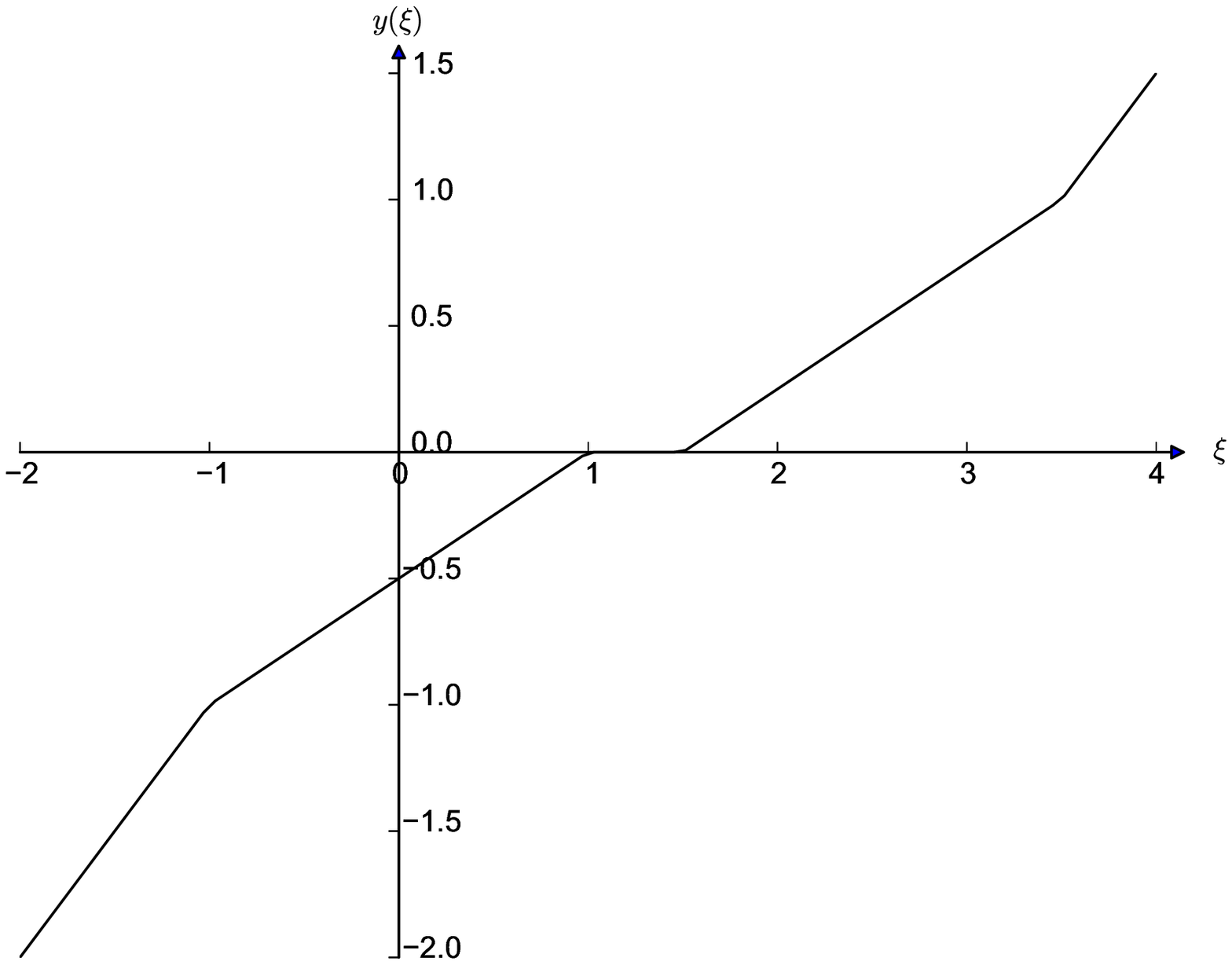}
\includegraphics[width=6cm]{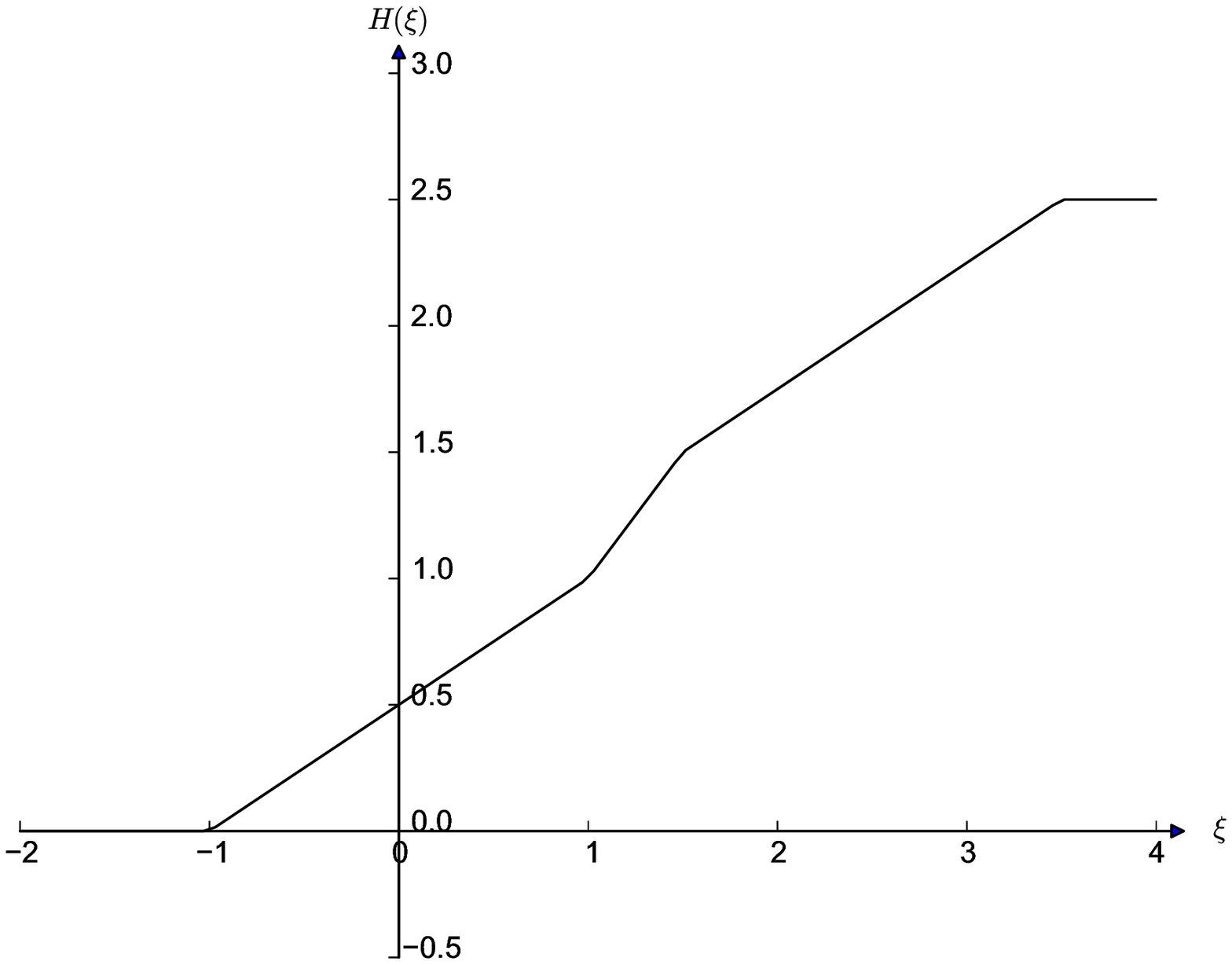}
\caption{A plot of the functions $F$, $y$, and $H$ in Example \ref{basic example}. Note that the jump in $F$ corresponds to the flat interval in $y$.}
\end{figure}
We also need a mapping $M:\F\rightarrow\D$ that takes us back from Lagrangian to Eulerian variables. The definition is similar to \cite[Theorem 4.10]{GHR}.
\begin{definition}
\label{definition M}
Let the mapping $M:\F\rightarrow\D$ be defined by $M(y,U,H,r) = (u,\rho,\mu)$ where
\begin{subequations}
\begin{align}
u(x) &= U\big(y(\xi)\big),\qquad \text{for some } \xi \text{ such that }x = y(\xi),\\
\rho\:\d x &= y_{\#}(r\:\d\xi),\\
\mu &= y_{\#}(H_{\xi}\:\d\xi).
\end{align}
\end{subequations}
The notation $f_{\#}(\mu)$ denotes the push-forward of the measure $\mu$ by the measurable function $f$, i.e.\ $f_{\#}(\mu)(A) = \int_{f^{-1}(A)}\:\d\mu$ for all measurable sets $A$.
\end{definition}
Since there are instances where $y_{\xi}=0$ on some interval, see Example \ref{basic example}, one might encounter difficulties when trying to invert $y$. Also it is not clear that the range of $M$ is $\D$. It is therefore necessary to prove that the mapping is well defined.
\begin{proposition}
\label{proposition M}
The mapping $M$ in Definition \ref{definition M} is well defined.
\end{proposition}
\begin{proof}
The proof follows closely those of \cite[Theorem 3.11]{HR07} and \cite[Theorem 4.10]{GHR}, but here the spaces $\D$ and $\F$ are different. The differences are that here $u,U\in L^{\infty}(\R)$, while \cite[Theorem 3.11]{HR07} and \cite[Theorem 4.10]{GHR} use $u,U\in L^2(\R)$. Let $X=(y,U,H,r)\in\F $ and $(u,\rho,\mu) = M(X)$. We prove that $u\in L^{\infty}(\R)$ since everything else is covered in the proof of \cite[Theorem 4.10]{GHR}. Since $u$ is a well defined function we must have that $u$ is bounded by $\|U\|_{\infty}$.
\end{proof}
\begin{example}
Let $X\in\F_0$ be given by $L(u,\rho,\mu)$ in Example \ref{basic example}. Then $M(X) = (u,\rho,\mu)$. To see that this is the case we compute the distribution function of the measure we get by applying $M$ to $X$. Let $x\in\R$ and $M(X)=(\bar{u},\bar{\rho},\bar{\mu})$, then
\begin{align*}
\bar{\mu}\left((-\infty,x]\right) &= \int_{y^{-1}\left((-\infty,x]\right)}H_{\xi}(\xi)\:\d\xi \\
	&=
		\begin{cases}
		0, & x \leq -1,\\
		\int_{-1}^{2x+1}\frac 12\:\d\xi = x+1, & -1\leq x < 0,\\
		\int_{\frac 32}^{2x+\frac 32}\frac 12\:\d\xi+\frac 12+1 = x+\frac 32, & 0\leq x\leq 1,\\
		\frac 52, & 1\leq x,
		\end{cases}
\end{align*}
which is equal to the distribution function $F$ of $\mu$ from Example \ref{basic example} plotted in Figure \ref{figure example}.
\end{example}
In the Lagrangian formulation there are four unknowns, while in the Eulerian there are only three. Hence it is not surprising that there is some redundancy in the Lagrangian formulation. Indeed there are equivalence classes in $\F$ such that all elements in an equivalence class map to the same element in $\D$. The equivalence classes are determined by a group of relabeling functions, $G$, and a relabeling operator $\bullet$.
\begin{definition}
\label{definition_G}
We define the group $G$ and the group action $\bullet$ of $G$ on $\F$ as follows.
\begin{itemize}
\item[(i)] Define $G$ as the group of homeomorphisms $f:\R\rightarrow\R$ such that both $f - \Id \in W^{1,\infty}(\R)$, $f^{-1}-\Id \in W^{1,\infty}(\R)$, and $f_{\xi}-1\in L^2(\R)$.
\item[(ii)] Define a group action $\bullet:\F\times G \rightarrow \F$ by $(X,f)\mapsto (y\circ f, U\circ f, H\circ f, (r\circ f) \cdot f')=X\bullet f$.
\end{itemize}
\end{definition}
The result is that $X$ and $X\bullet f$ is mapped to the same element in $\D$. For more details on the group $G$ and the group action $\bullet$, see \cite[Proposition 4.5]{GHR} and \cite[Proposition 3.4]{HR07}.
\begin{proposition}
\label{M and f}
Let $f\in G$ and $X\in\F$, then $M(X\bullet f) = M(X)$.
\end{proposition}
\begin{proof}
This proof is similar to that of \cite[Theorem 3.11]{HR07}. Let $X\in\F$ and $f\in G$ be given and let $(u,\rho,\mu) = M(X)$, and $(\bar{u},\bar{\rho},\bar{\mu})=M(X\bullet f)$, respectively. For a proof that $u=\bar{u}$ and $\mu = \bar{\mu}$ see \cite[Theorem 3.11]{HR07}. We prove that $\rho\:\d x = \bar{\rho}\:\d x$ in the sense of measures. Let $A\subseteq \R$ be of finite measure, and recall that $\rho\:\d x = y_{\#}(r\:\d\xi)$. Then
\begin{align}
\bar{\rho}\:\d x (A) &= (y\circ f)_{\#}( r\circ f f_{\xi}\:\d\xi)(A) \nonumber \\
	& = \int_{(y\circ f)^{-1}(A)}r\circ f(\xi)f_{\xi}(\xi)\:\d\xi.
\end{align}
Since $f$ is invertible and Lipschitz continuous, and $r\in L^1_{loc}(\R)$ we can change variables to obtain
\begin{align}
\label{rho is bar rho}
\int_{(y\circ f)^{-1}(A)}r\circ f(\xi)f_{\xi}(\xi)\:\d\xi &= \int_{f\circ (y\circ f)^{-1}(A)}r(\xi)\d\xi \nonumber \\
	&= \int_{y^{-1}(A)}r(\xi)\:\d\xi = y_{\#}(r\:\d\xi)(A).
\end{align}
In the last equality we have used that 
\begin{equation}
f(\{\xi\mid y\circ f(\xi)\in A\}) = \{f(\xi)\mid y\circ f(\xi)\in A\} = \{\xi\mid y(\xi)\in A\},
\end{equation}
since $f$ is onto, continuous, and strictly increasing. Equation \eqref{rho is bar rho} implies that $\rho = \bar{\rho}$ almost everywhere.
\end{proof}
We need that when one maps an element of Eulerian coordinates to Lagrangian coordinates and back that one should end up with the same element. In lieu of the previous propositions the converse cannot hold for $\F$. If we restrict $M$ to $\F_0$, however, we get that the composition $L\circ M$ is the identity function on $\F_0$.
\begin{lemma}[{\cite[Theorem 3.12]{HR07}}]
\label{lemma M and L}
The functions $L$ and $M$ satisfy
\begin{align}
M\circ L &= \Id_{\D},\\
L\circ M &= \Id_{\F_0},
\end{align}
when $M$ is restricted to $\F_0$. 
\end{lemma}

\section{Existence of solutions}\label{section existence}
In the previous section we saw that the space of Eulerian coordinates could be represented by Lagrangian coordinates. We now want to reformulate the initial value problem of \eqref{hunter-saxton-system} in Lagrangian coordinates. In this section we motivate the system \eqref{lagrangian system}, and show the existence of solutions for this system. Then we define conservative weak solutions of \eqref{hunter-saxton-system}, and show that we can construct such solutions by mapping the initial data from Eulerian to Lagrangian coordinates, solve \eqref{lagrangian system}, and map the solution back to Eulerian coordinates.

Due to the fact that a finite amount of energy accumulates in a point at wave breaking, we replace $(u_x^2+\rho^2) \:\d x$ by a measure $\mu$ such that $\mu_{ac} = (u_x^2+\rho^2)\:\d x$. Then the conservation law \eqref{conservation law} reads 
\begin{equation}
\label{conservation law mu}
\mu_t + (u\mu)_x = 0.
\end{equation}
Let $\xi\in\R$ and define $y$ by
\begin{equation}
\frac{\d}{\d t}y(\xi,t) = u(y(\xi,t),t),\quad y(\xi,0) = y_0(\xi).
\end{equation}
Then we define $U(\xi,t) = u(y(\xi,t),t)$, and for $y(\xi,t) \in \left(\mathrm{supp}\: \mu_{s}(t)\right)^c$, we define $H(\xi,t) = \mu\big((-\infty,y(\xi,t)]\big)$. We assume that $y(\xi,t)\in \left(\mathrm{supp}\:\mu_{s}(t)\right)^c$ and get
\begin{equation}
\frac{\d}{\d t}U(\xi,t) = u_t(y(\xi,t),t) + U(\xi,t)u_x(y(\xi,t),t) = \frac 12 H(\xi,t)-\frac 14 H_{\infty},
\end{equation}
where $U(\xi,0) = u(y_0(\xi),0)$, and we have used the equation for $u$ in \eqref{hunter-saxton-system}. Since $y(\xi,t)\in \left(\mathrm{supp}\:\mu_{s}(t)\right)^c$ we have from the conservation law \eqref{conservation law mu} that
\begin{align}
\frac{\d}{\d t}H(\xi,t) &= \mu_t\left((-\infty,y(\xi,t)]\right) + \frac{\d\mu}{\d x}(y(\xi,t),t)U(\xi,t)\nonumber\\
	&= -U(\xi,t)\frac{\d\mu}{\d x}(y(\xi,t),t) + U(t)\frac{\d\mu}{\d x}(y(\xi,t),t) = 0,
\end{align}
where $H(\xi,0) = \mu\big((-\infty,y_0(\xi)]\big)$. Since $\rho$ is a conserved variable it is natural to look at $\rho(x)\:\d x = \rho(y(\xi))y_{\xi}(\xi)\:\d\xi$. Define now $r(t) = \rho(y(t),t)y_{\xi}(t)$, then
\begin{align}
\frac{\d}{\d t}r(t) &= \left(\rho_t(y(t),t)+ \rho_x(y(t),t)U(t)\right)y_{\xi}(t) + \rho(y(t),t)U_{\xi}(t)\nonumber\\ 
	&= -\rho(y(t),t)U_{\xi}(t)+\rho(y(t),t)U_{\xi}(t) = 0,\\
r(0) &= \rho(y_0,0)y_{0\xi},
\end{align}
where we have used the equation for $\rho$ in \eqref{hunter-saxton-system}. The derivation assumed that $y(t)$ was outside the support of the singular part of $\mu(t)$, but we will extend the system to all of $\R$.

The system \eqref{lagrangian system} can be solved explicitly. We will be interested in the initial value problem with initial values in $\F_0$.
\begin{proposition}
The solution of the system \eqref{lagrangian system} with initial data
\begin{subequations}
\label{lagrangian solutions proposition}
\begin{align}
y|_{t=0} &= y_0,\\
U|_{t=0} &= U_0,\\
H|_{t=0} &= H_0,\\
r|_{t=0} &= r_0,
\end{align}
\end{subequations}
in $\F_0$, is given by
\begin{subequations}
\label{lagrangian solutions}
\begin{align}
y(\xi,t) &= \frac 14\big(H_0(\xi)-\frac 12 H_{\infty}\big)t^2 + U_0(\xi)t + y_0(\xi),\\
U(\xi,t) &= \frac 12\big(H_0(\xi)-\frac 12 H_{\infty}\big)t + U_0(\xi),\\
H(\xi,t) &= H_0(\xi),\\
r(\xi,t) &= r_0(\xi),
\end{align}
where $H_{\infty} = \|H\|_{\infty} = \lim_{\xi\rightarrow\infty}H(\xi)$.
\end{subequations}
\end{proposition}
\begin{proof}
To find the solutions we integrate \eqref{lagrangian system} with respect to $t$, starting with the equations for $r$ and $H$, and then proceed to $U$ and finally $y$. Uniqueness follows from the linearity of the system.
\end{proof}
\begin{example}
Let $X=L(u,\rho,\mu)$ be as in Example \ref{basic example}. Then $S_t(X)$ equals
\begin{align*}
y(\xi,t) &=
	\begin{cases}
	-\frac{5}{16}t^2 + t + \xi,& \xi\leq -1,\\
	\frac 18(\xi-\frac 32)t^2-\frac 12(\xi-1)t+\frac 12(\xi-1),& -1\leq\xi\leq 1,\\
	\frac 14(\xi-\frac 54)t^2, & 1\leq \xi \leq \frac 32,\\
	\frac 18(\xi-1)t^2 + \frac 12(\xi-\frac 32), &\frac 32\leq\xi\leq\frac 72,\\
	\frac{5}{16}t^2+\xi-\frac 52,&  \frac 72\leq \xi,
	\end{cases}\\
U(\xi,t) &= 
	\begin{cases}
	-\frac 58t + 1, & \xi\leq -1\\
	\frac 14(\xi - \frac 32)t - \frac 12(\xi-1), & -1\leq \xi\leq 1,\\
	\frac 12(\xi-\frac 54)t, & 1\leq \xi\leq \frac 32,\\
	\frac 14(\xi-1)t, & \frac 32\leq \xi \leq \frac 72,\\
	\frac 58 t, & \frac 72 \leq \xi,
	\end{cases}\\
H(\xi,t) &= H(\xi,0),\\
r(\xi,t) &= r(\xi,0).
\end{align*}
\end{example}
\begin{theorem}
\label{solution_operator_lagrangian}
The solution operators $S_t:\mathcal{F}\rightarrow\mathcal{F}$ constitute a semigroup. Furthermore the semigroup is Lipschitz continuous in $\F$ in the sense that for $X,\bar{X}\in\F$ there holds
\begin{equation}
\label{lipschitz in F}
\|S_t(X)-S_t(\bar{X})\|_B \leq (\frac 12t^2+t+1)\|X-\bar{X}\|_B.
\end{equation}
\end{theorem}
\begin{proof}
Let $X_0=(y_0,U_0,H_0,r_0)$ denote the initial data and $S_t(X_0)= X(t) =\big(y(t),U(t),H(t),r(t)\big)$ the solution of \eqref{lagrangian system} at $t$. We need to show that the solution is in $\F$. For each $t$ we have that $H$ is bounded by $\|H_0\|_{\infty}$, $|U|$ by $\frac{1}{4}\|H_0\|_{\infty}t+\|U_0\|_{\infty}$ and $|y|$ by $\frac{1}{8}\|H_0\|_{\infty}t^2+ \|U_0\|_{\infty}t+\|y_0\|_{\infty}$. Since the solutions \eqref{lagrangian solutions} are linear combinations of the initial data plus a constant we can differentiate the solutions with respect to $\xi$. If we differentiate the solutions \eqref{lagrangian solutions} with respect to $\xi$ we obtain
\begin{subequations}
\label{lagrangian solutions differentiated}
\begin{align}
y_{\xi}(\xi,t) &= \frac 14 H_{0\xi}(\xi)t^2 + U_{0\xi}(\xi)t+y_{0\xi}(\xi),\\
U_{\xi}(\xi,t) &= \frac 12 H_{0\xi}(\xi)t + U_{0\xi}(\xi),\\
H_{\xi}(\xi,t) &= H_{0\xi}(\xi).
\end{align}
\end{subequations}
We have the following estimates
\begin{align}
\label{derivative bound}
r(t) &= r_0,\nonumber\\
0\leq H_{\xi}(t)&\leq H_{0\xi},\nonumber\\
|U_{\xi}(t)| &\leq \frac{1}{2}H_{0\xi}t+|U_{0\xi}|,\nonumber\\
|y_{\xi}(t)-1| &\leq \frac{1}{4}H_{0\xi}t^2 + |U_{0\xi}|t+|y_{0\xi}-1|,
\end{align}
which are square integrable and bounded. Furthermore, $y-\Id ,U,H\in W^{1,\infty}(\R)$ as this holds for the initial data, and for each $t$ the solutions are linear combinations of the initial data. Thus property $(i)$ in Definition \ref{def_F} is proved. Consider now
\begin{align}
\label{(iii) holds in t}
U_{\xi}^2 &= \frac{1}{4}H_{0\xi}^2t^2 + U_{0\xi}H_{0\xi}t + U_{0\xi}^2, \nonumber\\
r^2 &= r_0^2 , \nonumber\\
y_{\xi}H_{\xi} &= \frac{1}{4}H_{0\xi}^2t^2 +U_{0\xi}H_{0\xi}t + U_{0\xi}^2 + r_0^2,
\end{align}
where it has been used that $H_{0\xi}y_{0\xi} = U_{0\xi}^2+r_0^2$. The above proves that $(iii)$ in Definition \ref{def_F} holds. Non-negativity of $y_{\xi}$ follows from $H_{\xi} = H_{0\xi}$ being non-negative, and thus $y_{\xi}$ has to be non-negative due to \eqref{(iii) holds in t}. We show that there exists a $c(t)$ dependent on $t$ such that $(y(t)+H(t))_{\xi}\geq c(t)>0$. By assumption it holds for $t=0$ with a constant $c(0)$, and since for each $\xi$ the functions $y_{\xi}$ and $H_{\xi}$ are continuous in $t$ it will hold on some interval $[0, T(\xi))$. We choose $T(\xi)$ to be the maximal time for which it holds. Then for $t\in [0,T(\xi))$ we have
\begin{align}
\frac{\d}{\d t}\frac{1}{y_{\xi}+H_{\xi}} &= -\frac{U_{\xi}}{(y_{\xi}+H_{\xi})^2}\nonumber\\
	&\leq \frac{1}{y_{\xi}+H_{\xi}}\frac{|U_{\xi}|}{y_{\xi}+H_{\xi}} \nonumber\\
	&\leq \frac{1}{y_{\xi}+H_{\xi}}\frac{\sqrt{y_{\xi}H_{\xi}}}{y_{\xi}+H_{\xi}}\nonumber\\
	&\leq \frac 12\frac{1}{y_{\xi}+H_{\xi}}.
\end{align}
By Gr\" onwall's inequality
\begin{equation}
\label{y plus H bound}
\frac{1}{y_{\xi}+H_{\xi}}(t) \leq \frac{1}{c(0)}e^{\frac 12t},
\end{equation}
where $c(0)$ is the initial constant, and $T(\xi)$ can be chosen to be arbitrarily big, which shows that property $(ii)$ in Definition \ref{def_F} holds. We prove \eqref{lipschitz in F}. Since the time evolution of the derivatives are linear in derivatives of $X$ we can use \eqref{lagrangian solutions differentiated} to find
\begin{align}
\|y_{\xi}(t)-\bar{y}_{\xi}(t)\|_2 &\leq \frac 14 \|H_{0\xi}-\bar{H}_{0\xi}\|_2 t^2 + \|U_{0\xi}-\bar{U}_{0\xi}\|_2t+\|y_{0\xi}-\bar{y}_{0\xi}\|_2,\\
\|U_{\xi}(t)-\bar{U}_{\xi}(t)\|_2 &\leq \frac 12 \|H_{0\xi}-\bar{H}_{0\xi}\|_2 t + \|U_{0\xi}-\bar{U}_{0\xi}\|_2.
\end{align}
In the $L^{\infty}$-part the term $H_{\infty}$ makes the solution operator nonlinear. However, we can bound
\begin{equation}
\label{estimate H lipschitz}
\|H-\bar{H} -\frac 12 H_{\infty}+\frac 12 \bar{H}_{\infty}\|_{\infty} \leq \frac 32\|H-\bar{H}\|_{\infty},
\end{equation}
and thus
\begin{align}
\|y(t)-\bar{y}(t)\|_{\infty} &\leq \frac 12 \|H_0-\bar{H}_0\|_{\infty} t^2 + \|U_0-\bar{U}_0\|_{\infty}t+\|y_0-\bar{y}_0\|_{\infty},\\
\|U(t)-\bar{U}(t)\|_{\infty} &\leq \|H_0-\bar{H}_0\|_{\infty} t + \|U_0-\bar{U}_0\|_{\infty}.
\end{align}
Since $H$ and $r$ do not change in $t$, the estimate is proved. The estimate \eqref{estimate H lipschitz} implies that the right-hand side of \eqref{lagrangian system} is Lipschitz in $\F$, and thus $S_t$ satisfies the semigroup property.
\end{proof}
We define the map $T_t:\D\rightarrow\D$ by
\begin{equation}
\label{def_T_t}
T_t = M\circ S_t\circ L.
\end{equation}
\begin{definition}
\label{definition_solution}
A triple $(u,\rho,\mu)\in\D$ is said to be a conservative weak solution of \eqref{hunter-saxton-system} if for any test function $\phi\in C_0^{\infty}(\R\times [0,\infty))$,
\begin{subequations}
\begin{align}
&\int\limits_0^{\infty}\int\limits_{\R}\big(u\phi_t+\frac{1}{2}u^2\phi_x+\frac 14\big(\int\limits_{-\infty}^x \:\d \mu(t)- \int\limits_{x}^{\infty}\:\d \mu(t)\big)\phi\:\d x\d t =-\int\limits_{\R}\big(u\phi|_{t=0}\big)\:\d x,\\
&\int\limits_0^{\infty}\int\limits_{\R}\big(\rho\phi_t + \rho u\phi_x\big)\:\d x\d t = -\int\limits_{\R}(\rho\phi)|_{t=0}\:\d x,\\
&\int\limits_0^{\infty}\int\limits_{\R}(\phi_t + u\phi_x)\:\d\mu(t)\d t = \int\limits_{\R}\phi|_{t=0}\d\mu|_{t=0}\:\d t,
\end{align}
\end{subequations}
and in addition
\begin{equation}
\mu(t)(\R) = \mu_0(\R),
\end{equation}
holds for all $t\geq 0$.
\end{definition}
We can now use the operator $T_t$ defined by \eqref{def_T_t} to construct conservative weak solutions of \eqref{hunter-saxton-system}.
\begin{theorem}
\label{theorem_existence}
The operator $T_t: (u,\rho,\mu)\mapsto \left(u(t),\rho(t),\mu(t)\right)$ maps an initial value to a conservative weak solution of \eqref{hunter-saxton-system} in the sense of Definition \ref{definition_solution}. 
\end{theorem}
\begin{proof}
The idea of the proof is to first map the initial data in $\D$ to $\F_0$ by $L$, then solve the problem there and for each $t$ map the solution to $\D$ by $M$. Change of variables must be done on the set $\{\xi\mid y_{\xi}(\xi)>0\}$, but since both $U_{\xi}$ and $r$ equals zero almost everywhere on the complement we can integrate over $\R$ when we change variables in $u$ and $\rho$. Since $\mu(t) = y(t)_{\#}\left(H_{\xi}\:\d\xi\right)$ it holds that $\mu\left((-\infty,x]\right) = \sup\left\{H(\xi)\mid y(\xi) = x\right\}$. Thus
\begin{align*}
&\int\limits_0^{\infty}\int\limits_{\R}\big(u\phi_t+\frac{1}{2}u^2\phi_x+\frac 14\big(\int\limits_{-\infty}^x \:\d \mu(t)- \int\limits_{x}^{\infty} \:\d \mu(t)\big)\phi\:\d x\d t \nonumber\\ &= \int\limits_0^{\infty}\int\limits_{\R}U(\phi_t\circ y+\frac 12 U\phi_x\circ y)y_{\xi}(\xi)\:\d\xi\d t\nonumber\\ &\quad+ \frac 12\int\limits_0^{\infty}\int\limits_{\R}\left(\int\limits_{-\infty}^{y(\xi,t)} \:\d\mu(t)- \frac 12\int\limits_{-\infty}^{\infty} \:\d\mu(t)\right)\phi\circ y y_{\xi}\:\d\xi\d t \nonumber\\
	&= \int\limits_0^{\infty}\int\limits_{\R}\bigg[ U(\xi,t)\big(\frac{\d}{\d t}\phi(y(\xi,t),t)-\frac 12 U\phi_x(y(\xi,t),t)\big)\nonumber\\ &\quad -\frac 12\left(H(\xi,t)-\frac 12 H_{\infty}\right)\phi\circ y \bigg]y_{\xi}\:\d\xi\d t \nonumber\\
	&= \int\limits_0^{\infty}\int\limits_{\R}\big[ \frac{\d}{\d t}(U\phi\circ y) - \frac 12 U^2\phi_x\circ y\big]y_{\xi}\:\d\xi\d t \nonumber\\
	&= \int\limits_0^{\infty}\int\limits_{\R}\big[\frac{\d}{\d t}(U\phi\circ y y_{\xi}) - (\frac 12 U^2\phi\circ y)_{\xi}\big]\:\d\xi\d t \nonumber\\
	&= -\int\limits_{\R} u_0(x)\phi(x,0)\:\d x.
\end{align*}
The equation for $\rho$ is treated in the same way,
\begin{align*}
\int\limits_0^{\infty}\int\limits_{\R}\rho\phi_t + (\rho u)\phi_x\:\d x\d t &= \int\limits_0^{\infty}\int\limits_{\R}\big(\phi_t +  u\phi_x\big)\rho\:\d x\d t \nonumber\\
&= \int\limits_0^{\infty}\int\limits_{\R}\big(\phi_t\circ y +  U\phi_x\circ y\big)r\:\d\xi\d t \nonumber\\
&= \int\limits_0^{\infty}\int\limits_{\R}\frac{\d}{\d t}\phi(y(\xi,t),t)r(\xi,t)\:\d\xi\d t \nonumber\\
&= \int\limits_0^{\infty}\int\limits_{\R}\frac{\d}{\d t}\big(\phi(y(\xi,t),t)r(\xi,t)\big)\:\d\xi\d t \nonumber\\
&= -\int\limits_{\R}\phi(y_0(\xi),0)r_0(\xi)\:\d\xi\nonumber\\
&= -\int\limits_{\R}\phi(x,0)\rho_0(x)\:\d x.
\end{align*}
The equality on conservation of $\mu$ is proved
\begin{align*}
\int\limits_0^{\infty}\int\limits_{\R}(\phi_t+u\phi_x)\:\d\mu(t)\d t &= \int\limits_0^{\infty}\int\limits_{\R}(\phi_t+u\phi_x)\circ y H_{\xi}\:\d\xi\d t \nonumber\\
	&= \int\limits_0^{\infty}\int\limits_{\R}H_{\xi}\frac{\d}{\d t}(\phi\circ y)\:\d\xi\d t \nonumber\\
	&= -\int\limits_{\R} H_{\xi}(\phi\circ y)|_{t=0}\:\d\xi \nonumber\\
	&= - \int\limits_{\R} \phi|_{t=0}\:\d\mu|_{t=0}.
\end{align*}
For every $t$ we have that $y(t)^{-1}(\R) = \R$, and thus
\begin{equation*}
\mu(t)(\R) = \int\limits_{y(t)^{-1}(\R)}H_{\xi}(\xi,t)\:\d\xi = \int\limits_{\R}H_{0\xi}\:\d\xi = \mu_0(\R).
\end{equation*}
\end{proof}
\begin{example}
\label{example solution operator}
Let $(u_0,\rho_0,\mu_0)\in\D$ be as in Example \ref{basic example}. Then $T_t(u_0,\rho_0,\mu_0)$ is given by
\begin{align}
u(x,t) &=
	\begin{cases}
	-\frac 58t + 1, & x \leq -\frac{5}{16}t^2+t-1,\\
	-\frac{x+\frac{1}{16}t^2}{1-\frac 12 t}-\frac 18t, & -\frac{5}{16}t^2+t-1\leq x\leq -\frac{1}{16}t^2,\\
	\frac{2x}{t}, & -\frac{1}{16}t^2\leq x\leq \frac{1}{16} t^2,  \\
	\frac 12\frac{x+\frac 14}{\frac 14 t^2+1}t, & \frac{1}{16}t^2\leq x\leq \frac{5}{16}t^2+1,\\
	\frac 58 t, & \frac{5}{16}t^2 + 1\leq x,
	\end{cases}\\
\rho(x,t) &= 
	\begin{cases}
	0, & x < \frac{1}{16}t^2,\\
	\frac{1}{\frac 14t^2+1}, & \frac{1}{16}t^2\leq x\leq \frac{5}{16}t^2+1,\\
	0, & \frac{5}{16}t^2 + 1 < x,
	\end{cases}\\
\mu(t) &= \mu(t)_{ac} + \frac 12\delta_{(0,0)}(x,t) + \delta_{(-\frac 14,2)}(x,t).
\end{align}
\end{example}
\begin{remark}
Note that in Example \ref{example solution} and \ref{example solution operator} the functions $u(x,0)$ and $\rho(x,0)$ coincide, while in Example \ref{example solution} and Example \ref{example solution operator}, respectively, we have $\mu(0)=\big(u_x(x,0)^2+\rho(x,0)^2\big)\:\d x$ and $\mu(0)=\big(u_x(x,0)^2+\rho(x,0)^2\big)\:\d x + \frac 12\delta_0$. The solutions in Example \ref{example solution} and \ref{example solution operator} differ, hence it is important to include the measure $\mu$ in the description of global solutions and the definition of the Lipschitz metric.
\end{remark}

\section{The Lipschitz metric}\label{section metric}
In this section we construct a metric on $\F_0$ that renders the flow Lipschitz continuous with respect to initial data. We saw in Proposition \ref{M and f} that the mapping $M:\F\rightarrow\D$ is relabeling invariant. An important fact is that relabeling commutes with the solution operator $S_t$.
\begin{proposition}
\label{Semigroup commutes with relabeling}
For any $X\in\F, f\in G$ it holds that $S_t(X\bullet f) = S_t(X)\bullet f$.
\end{proposition}
\begin{proof}
Any component of $S_t(X)$ is a linear combination of components of $X$.
\end{proof}
The problem of creating a metric directly on $\F$ is that $X$ and $\bar{X}$ may correspond to the same solution in Eulerian coordinates even if $X\neq\bar{X}$. Thus we will try to compare solutions in $\F_0$, but we must prove that we can reach $\F_0$ from all of $\F$ via the group action $\bullet$ from Definition \ref{definition_G} (ii).
\begin{definition}
Define the map $\Pi:\F\rightarrow\F_0$ by 
\begin{equation}
\Pi X = X\bullet (y+H)^{-1}, \quad X\in\F.
\end{equation}
To ease the notation we write $\Pi X$, despite the fact that $\Pi$ is not a linear operator.
\end{definition}
For the map $\Pi$ to be a relabeling we need that $y+H\in G$.
\begin{proposition}
\label{y plus H in G}
Let $X\in\F_0$, and $X(t)=S_t(X)$. Then for all $t\geq 0$ we have that $y(t)+H(t)\in G$, and $e^{-\frac 12t}\leq y_{\xi}(t) + H_{\xi}(t)\leq \frac 14t^2+t+1$ for almost every $\xi\in\R$.
\end{proposition}
\begin{proof}
Let $y+H = f$, we show that $f\in G$. From the definition of $\F$ we have that $f-\Id\in W^{1,\infty}(\R)$ with $f_{\xi}-1\in L^2(\R)$, with $c\leq f_{\xi}\leq C$ for some positive numbers $c$ and $C$. Thus there exists a Lipschitz continuous inverse $f^{-1}$ such that $f^{-1}-\Id\in W^{1,\infty}(\R)$. Hence $f\in G$. The lower bound on $y_{\xi}(t)+H_{\xi}(t)$ is given by \eqref{y plus H bound}, while the upper bound is a result of the time evolution and $|y_{0\xi}|$, $|H_{0\xi}|$, and $|U_{0\xi}|$ all being less than $1$.
\end{proof}
The metric induced by $\|\:\cdot\:\|_B$ will unfortunately give a positive distance between $X$ and $X\bullet f$, even though they will map to the same element in Eulerian coordinates via the mapping $M$. One could potentially restrict attention to the class $\F_0$ by comparing $X\bullet (y+H)^{-1}$ and $\bar{X}\bullet (\bar{y}+\bar{H})^{-1}$, but it has proven difficult to control the $t$-dependence of $(y+H)^{-1}$. Instead we minimize the distance over all possible relabelings. Following \cite{LipCHline} we define $J:\F\times \F\rightarrow \R$ by
\begin{equation}
J(X,\bar{X}) = \inf\limits_{f,g\in G}\big(\|X \bullet f - \bar{X}\|_B+ \|X-\bar{X}\bullet g\|_B\big).
\end{equation}
If we instead tried $\|X\bullet f-\bar{X}\bullet g\|_B$ we would not be able to separate $r$ from $-r$ as the next example illustrates.
\begin{example}
\label{example}
Let $r = \mathbf{1}_{[0,1]}$, and $y,U,H$ such that $X=(y,U,H,r)\in\F_0$. Then $\bar{X}=(y,U,H,-r)\in\F_0$ as well. Let $0<\epsilon<1$ and define $f^{\epsilon}\in G$ by
\begin{equation}
f^{\epsilon}(\xi) =
\begin{cases}
\xi,  &\xi < 0,\\
\epsilon\xi,  & 0\leq \xi < \frac{1}{\epsilon},\\
\xi+\left(1-\frac{1}{\epsilon}\right), & \frac{1}{\epsilon} \leq \xi.
\end{cases}
\end{equation}
Then for any $t\geq 0$, all components of $S_t(X)$ and $S_t(\bar{X})$ except $r,\bar{r}$ are equal, and thus,
\begin{equation}
\|S_t(X)\bullet f^{\epsilon}-S_t(\bar{X})\bullet f^{\epsilon}\|_B \leq \left(\int_0^{\frac 1\epsilon} 4\epsilon^2\:\d\xi\right)^{1/2} \leq 2\sqrt{\epsilon},
\end{equation}
and hence the infimum over all $f,g\in G$ must equal zero.
\end{example}
In the way we have defined $J$ here we avoid the scenario in Example \ref{example} since we cannot make both $r$ and $\bar{r}$ small at the same time. Now, $J$ will not separate $X$ and $X\bullet f$, and behaves as expected with $t$. However $J$ is not a metric, as the triangle inequality fails. One can salvage a metric by taking the infimum over sums in $J$ over finite sequences in $\F_0$.\footnote{This idea is due to A. Bressan.}
\begin{definition}
Let $X,\bar{X}\in\F_0$, then define $d:\F_0\times\F_0\rightarrow \R$ by
\begin{equation}
d(X,\bar{X}) = \inf\sum\limits_{n=1}^N J(X_{n-1},X_n),
\end{equation}
where the infimum is taken over all finite sequences $\{X_n\}_{n=0}^N$ in $\F_0$ such that the endpoints $X_0$ and $X_N$ satisfy
\begin{align}
X &=X_0,\\
\bar{X} &= X_N.
\end{align}
\end{definition}
It is not at all clear that $d$ only vanishes when $X=\bar{X}$. The purpose of the next lemma is to assert that we have a positive lower bound on $d(X,\bar{X})$ when $X$ differs from $\bar{X}$.
\begin{lemma}[{\cite[Lemma 3.2]{LipCHline}}]
\label{lower_bound_d_lemma}
For any $X,\bar{X}\in\F_0$ we have
\begin{equation}
\label{lower_bound_d}
\|y-\bar{y}\|_{\infty}+\|U-\bar{U}\|_{\infty}+\|H-\bar{H}\|_{\infty} \leq 2d(X,\bar{X}).
\end{equation}
\end{lemma}
Lemma \ref{lower_bound_d_lemma} states that if the distance between $X$ and $\bar{X}$ equals zero, then $(y,U,H)$ and $(\bar{y},\bar{U},\bar{H})$ coincide. Still, $r$ and $\bar{r}$ could, in principle, differ. The next lemma shows that this cannot be the case, and consequently $d$ is a metric on $\F_0$.
\begin{lemma}[{\cite[A weaker form of Lemma 6.4]{LipCH}}]
Let $X,\bar{X}$ be in $\F_0$, then if $d(X,\bar{X}) = 0$ we have that $r=\bar{r}$.
\end{lemma} 
We need to estimate $J\big(\Pi S_t(X),\Pi S_t(\bar{X})\big)$ in terms of $J\big(S_t(X),S_t(\bar{X})\big)$.
\begin{lemma}
\label{lemma J relabeling f}
For $X,\bar{X}\in\F_0$, there holds
\begin{equation}
J\big(\Pi S_t(X),\Pi S_t(\bar{X})\big) \leq e^{\frac 12t} J\big(S_t(X),S_t(\bar{X})\big).
\end{equation}
\end{lemma}
\begin{proof}
The proof consists of two parts. First we show that for $f\in G$ with $\sqrt{\|f_{\xi}\|_{\infty}}\leq C$, with $C\geq 1$, it holds that for any $X,\bar{X}\in \F$
\begin{equation}
J(X\bullet f, \bar{X}) \leq C J(X,\bar{X}).
\end{equation}
Second we show that if $X\in\F_0$ then $(y(t)+H(t))^{-1}$ is in $G$ with $((y(t)+H(t))^{-1})_{\xi}\leq e^{\frac 12 t}$. The $L^{\infty}(\R)$ part of $J$ is invariant with respect to relabeling. It suffices to show that it holds for the $L^2(\R)$ part. To that end let $h$ and $\bar{h}$ belong to $L^2(\R)$, and let $g\in G$. Then by change of variables $f^{-1}(\eta) = \xi$, and denoted $\bar{f} = g\circ f^{-1}$,
\begin{align}
&\int_{\R}|h\circ f f_{\xi} - \bar{h}\circ g g_{\xi}|^2\d\xi\nonumber\\
&= \int_{\R}|h(\eta)f_{\xi}(f^{-1}(\eta)) - \bar{h}\circ \bar{f}(\eta) g_{\xi}(f^{-1}(\eta))|^2\frac{1}{f_{\xi}\left(f^{-1}(\eta)\right)}\:\d\eta \nonumber\\
	&= \int_{\R}|h(\eta) - \bar{h}\circ \bar{f}(\eta) \bar{f}_{\eta}(\eta)|^2f_{\xi}\circ f^{-1}(\eta)\:\d\eta \nonumber\\
	&\leq \|f_{\xi}\|_{\infty}\int_{\R}|h(\eta) - \bar{h}\circ \bar{f}(\eta) \bar{f}_{\eta}(\eta)|^2\d\eta,
\end{align}
where we have used that $\bar{f}_{\eta} = \frac{g_{\xi}\circ f^{-1}}{f_{\xi}\circ f^{-1}}$. We get the inequality
\begin{equation}
\inf_{g\in G}\|h\circ f f_{\xi} - \bar{h}\circ g g_{\xi}\|_2 \leq C\inf_{\bar{f}\in G}\|h-\bar{h}\circ \bar{f} \bar{f}_{\xi}\|_2,
\end{equation}
with $C=\sqrt{\|f_{\xi}\|_{\infty}}$. The part where one relabels $h\circ f f_{\xi}$ is fine since $X\bullet f\bullet g$ equals $X\bullet \tilde{f}$ for some $\tilde{f}\in G$, and infimum over $G$ is taken afterwards. Proposition \ref{y plus H in G} states that $y(t) + H(t) \in G$, and that $e^{-\frac 12 t} \leq y_{\xi}+H_{\xi} \leq (\frac 14 t^2+t+1)$. One can then invert the lower bound on $y_{\xi}+H_{\xi}$ and apply the first result twice to obtain the bound in the lemma.
\end{proof}
We are now ready to prove the Lipschitz theorem on $\F_0$.
\begin{theorem}
\label{lipschitz_theorem}
Let $X,\bar{X}\in\F_0$, then for all $t\geq 0$ it holds that
\begin{equation}
\label{lipschitz_equation}
d\left(\Pi S_t(X),\Pi S_t(\bar{X})\right) \leq e^{\frac 12 t}(\frac 12t^2+t+1)d(X,\bar{X}).
\end{equation}
\end{theorem}
\begin{proof}
Let $1>\epsilon>0$ and $X,\bar{X}\in\F_0$ be given and choose $\{X_n\}_{n=0}^N$, $\{f_n\}_{n=1}^N$, and $\{g_n\}_{n=0}^{N-1}$ such that $X_0=X, X_N=\bar{X}$ and $d(X,\bar{X})+\epsilon \geq \|X_n\bullet f_n-X_{n-1}\|_B+\|X_n-X_{n-1}\bullet g_{n-1}\|_B$. Then from the definition of $d$ we have
\begin{align}
d\big(\Pi S_t(X),\Pi S_t(\bar{X})\big) &\leq \sum_{n=1}^N J\big(\Pi S_t(X_n),\Pi S_t(X_{n-1})\big) \nonumber\\ 
	&\leq e^{\frac 12t}\sum_{n=1}^N J\big(S_t(X_n), S_t(X_{n-1})\big),
\end{align}
by Lemma \ref{lemma J relabeling f}. From the definition of $J$ we get 
\begin{align}
d\big(\Pi S_t(X),\Pi S_t(\bar{X})\big)	&\leq e^{\frac 12t}\sum_{n=1}^N J\big(S_t(X_n),S_t(X_{n-1})\big)\nonumber\\
	&\leq e^{\frac 12t}(\frac 12t^2+t+1)\sum_{n=1}^N\big(\|X_n\bullet f_n-X_{n-1}\|_B\nonumber\\ &\:\:\quad\qquad\qquad\qquad\qquad +\|X_n-X_{n-1}\bullet g_{n-1}\|_B\big)\nonumber\\
	&\leq e^{\frac 12t}(\frac 12t^2+t+1)\big(d(X,\bar{X})+\epsilon\big).
\end{align}
The inequality holds for each $\epsilon$ in the range $(0,1)$, which implies that
\begin{equation}
d\left(\Pi S_t(X(t)),\Pi S_t(\bar{X}(t))\right) \leq e^{\frac 12t}(\frac 12t^2+t+1)d(X,\bar{X}).
\end{equation}
\end{proof}
Since $\F_0$ is in one to one correspondance with $\D$ the metric $d$ on $\F_0$ induces a metric $d_{\D}$ on $\D$.
\begin{definition}
Define the metric $d_{\D}$ on $\D$ by
\begin{equation}
d_{\D}\big((u,\rho,\mu),(\bar{u},\bar{\rho},\bar{\mu})\big) = d\big(L(u,\rho,\mu),L(\bar{u},\bar{\rho},\bar{\mu})\big),
\end{equation}
for any $(u,\rho,\mu),(\bar{u},\bar{\rho},\bar{\mu})\in\D$.
\end{definition}
\begin{theorem}
The solution operator $T_t$ defined by \eqref{def_T_t} forms a Lipschitz continuous semigroup on $(\D,d_{\D})$ in the sense that for any  $(u,\rho,\mu),(\bar{u},\bar{\rho},\bar{\mu})\in\D$ the inequality
\begin{equation}
d_{\D}\big(T_t(u,\rho,\mu),T_t(\bar{u},\bar{\rho},\bar{\mu})\big) \leq e^{\frac 12t}(\frac 12 t^2+t+1)d_{\D}\big((u,\rho,\mu),(\bar{u},\bar{\rho},\bar{\mu})\big)
\end{equation}
holds.
\end{theorem}
\begin{proof}
The Lipschitz continuity is a corollary of Theorem \ref{theorem_existence} and Theorem \ref{lipschitz_theorem}. We prove the semi group property. From Proposition \ref{M and f} we have $M\circ\Pi = M$, hence
\begin{equation}
T_t = M\circ S_t\circ L = M\circ\Pi\circ S_t\circ L.
\end{equation}
From Theorem \ref{solution_operator_lagrangian} and Proposition \ref{Semigroup commutes with relabeling} we have
\begin{equation*}
T_tT_s = M\circ\Pi\circ S_t\circ L \circ M\circ\Pi \circ S_s\circ L = M\circ\Pi\circ S_t\circ \Pi\circ S_s\circ L = M\circ\Pi\circ S_{t+s}\circ L = T_{t+s}.
\end{equation*}
\end{proof}
\subsection*{Acknowledgements}
The author is grateful for discussions with Katrin Grunert, Helge Holden, and Xavier Raynaud.

\bibliographystyle{plain}
\bibliography{referanser}

\end{document}